\documentclass[12pt]{amsart}

\usepackage{amsmath}
\usepackage{amscd}
\usepackage{amssymb}
\usepackage{amsfonts}
\usepackage{amsthm}
\usepackage{enumerate}
\usepackage{hyperref}
\usepackage{color}
\usepackage{graphicx}
\theoremstyle{plain}
\newtheorem{thm}{Theorem}[section]

\newtheorem{lem}[thm]{Lemma}

\theoremstyle{definition}

\newtheorem{rem}[thm]{Remark}

\newtheorem{dfns-rems}[thm]{Definitions and Remarks}
\newtheorem{notas-rems}[thm]{Notations and Remarks}
\newtheorem{exmps-rems}[thm]{Examples and Remarks}

\makeatletter
\@namedef{subjclassname@2020}{%
  \textup{2020} Mathematics Subject Classification}
\makeatother

\begin{document}

% ------------------------------------------------------------------------

\title[Increasing normalized depth function]{An Increasing normalized depth function}

% ------------------------------------------------------------------------

\author[S. A. Seyed Fakhari]{S. A. Seyed Fakhari}

\address{S. A. Seyed Fakhari, Departamento de Matem\'aticas\\Universidad de los Andes\\Bogot\'a\\Colombia.}

\email{s.seyedfakhari@uniandes.edu.co}
% ------------------------------------------------------------------------

\begin{abstract}
Let $\mathbb{K}$ be a field and $S=\mathbb{K}[x_1,\ldots,x_n]$ be the polynomial ring in $n$ variables over $\mathbb{K}$. Assume that $I$ is a squarefree monomial ideal of $S$. For every integer $k\geq 1$, we denote the $k$-th squarefree power of $I$ by $I^{[k]}$. The normalized depth function of $I$ is defined as $g_I(k)={\rm depth}(S/I^{[k]})-(d_k-1)$, where $d_k$ denotes the minimum degree of monomials belonging to $I^{[k]}$. Erey, Herzog, Hibi and Saeedi Madani conjectured that for any squarefree monomial ideal $I$, the function $g_I(k)$ is nonincreasing. In this short note, we provide a counterexample for this conjecture. Our example in fact shows that $g_I(2)-g_I(1)$ can be arbitrarily large.
\end{abstract}

% ------------------------------------------------------------------------


% ------------------------------------------------------------------------


% ------------------------------------------------------------------------

\maketitle

%%%%%%%%%%%%%%%%%%%%%%%%%%%%%%%%%%%%%%%%%%%%%%%%%%%%%%%%%%%%%%%%%%%%%%%%%%

\section{Introduction} \label{sec1}

Let $\mathbb{K}$ be a field and $S=\mathbb{K}[x_1,\ldots,x_n]$ be the
polynomial ring in $n$ variables over $\mathbb{K}$. For any squarefree monomial ideal $I\subset S$ and for any positive integer $k$, the $k$-th squarefree power of $I$ denoted by $I^{[k]}$ is the ideal generated by the squarefree monomials belonging to $I^k$. In \cite{ehhs}, Erey, Herzog, Hibi and Saeedi Madani studied the depth of squarefree powers. They introduced the notion of normalized depth function as follows. Let $\nu(I)$ be the largest integer $k$ with $I^{[k]}\neq 0$. For each integer $k=1,2, \ldots, \nu(I)$, we denote the minimum degree of monomials belonging to $I^{[k]}$ by $d_k$. The normalized depth function of $I$ is the function $g_I:\{1, 2, \ldots, \nu(I)\}\rightarrow\mathbb{Z}_{\geq 0}$ defined by$$g_I(k)={\rm depth}(S/I^{[k]})-(d_k-1).$$ The same authors conjectured that for any squarefree monomial ideal $I$, the function $g_I(k)$ is nonincreasing. This conjecture is known to be true in special cases (see e.g., \cite{cfl}, \cite{ehhs}, \cite{fhh}). However, in the next section, we provide a class of ideals disproving the conjecture. Our example indeed shows that the difference $g_I(2)-g_I(1)$ can be arbitrarily large.

%%%%%%%%%%%%%%%%%%%%%%%%%%%%%%%%%%%%%%%%%%%%%%%%%%%%%%%%%%%%%%%%%%%%%%%%%%

\section{An example} \label{sec2}

In Theorem \ref{main}, we introduce a class of ideals $I$ showing that the normalized depth function $g_I(k)$ is not necessarily nonincreasing.

We recall that for any graph $G$ with vertex set $V(G)=\{1, 2, \ldots, n\}$ and edge set $E(G)$, its edge ideal is defined as$$I(G)=(x_ix_j\mid \{i, j\}\in E(G))\subset S.$$Moreover, a graph $G$ is said to be sequentially Cohen-Macaulay over $\mathbb{K}$ if $S/I(G)$ is sequentially Cohen-Macaulay (one may look at \cite[Chapter III]{s} for the definition of sequentially Cohen-Macaulay modules). We say that $G$ is a sequentially Cohen-Macaulay graph if it is sequentially Cohen-Macaulay over any field $\mathbb{K}$. A subset $U$ of $V(G)$ is called an independent subset of $G$ if there are no edges among the vertices of $U$. We say that a subset $C\subseteq V(G)$ is a {\it minimal vertex cover} of $G$ if, first, every
edge of $G$ is incident with a vertex in $C$ and, second, there is no
proper subset of $C$ with the first property. Note that $C$ is a minimal
vertex cover if and only if $V(G)\setminus C$ is a maximal independent subset of $G$. Moreover, it is known by \cite[Lemma 9.1.4]{hh'} that every minimal prime ideal of $I(G)$ is of the form $(x_i\mid i\in C)$ where $C$ is a minimal vertex cover of $G$. Since $I(G)$ is a radical ideal, it follows that the irredundant
primary decomposition of $I(G)$ is given by$$I(G)=\bigcap (x_i\mid i\in C),$$where the intersection is taken over all minimal
vertex covers $C$ of $G$.

We first need the following simple lemma.

\begin{lem} \label{tree}
Let $T$ be a tree with $n$ vertices. Then ${\rm depth}(S/I(T))$ is equal to the minimum size of a maximal independent subset of $T$.
\end{lem}

\begin{proof}
It is well-known that any tree is a sequentially Cohen-Macaulay graph (see e.g., \cite[Theorem 1.2]{fv}). Hence, it follows from \cite[Theorem 4]{f} (see also \cite[Corollary 3.33]{mv}) that ${\rm depth}(S/I(T))$ is equal to $n-h$, where $h$ denotes the maximum height of an associated prime of $I(T)$. Thus, using the primary decomposition of $I(T)$ given above, we deduce that $h$ is the maximum size of a minimal vertex cover of $T$. Therefore, $n-h$ is the minimum size of a maximal independent subset of $T$.
\end{proof}

We are now ready to present our example.

\begin{thm} \label{main}
Let $n\geq 6$ be an integer and consider the polynomial ring $S=\mathbb{K}[x_1,\ldots,x_n]$. For each integer $i$ with $1\leq i\leq n-4$, set $u_i:=x_1x_3x_{i+4}$. Also, set$$u_{n-3}:=x_1x_4x_5, \ \ \ \ u_{n-2}:=x_2x_3x_4 \ \ \ \ {\rm and} \ \ \ \ u_{n-1}:=x_2x_3x_6.$$Let $I$ be the squarefree monomial ideal generated by $u_1, u_2, \ldots, u_{n-1}$. Then
\begin{itemize}
\item[(i)] $g_I(1)=1$; and
\item[(ii)] $g_I(2)=n-6$.
\end{itemize}
In particular, $g_I(2)=g_I(1)+n-7$.
\end{thm}

\begin{proof}
(i) One can easily see that $\mathfrak{p}=(x_4, \ldots, x_n)$ is a minimal prime ideal of $I$. Thus,
\[
\begin{array}{rl}
{\rm depth}(S/I)\leq {\rm dim}(S/\mathfrak{p})=3.
\end{array} \tag{1} \label{1}
\]
Consider the following short exact sequence.
\begin{align*}
& 0\longrightarrow \frac{S}{(I:x_3)}\longrightarrow \frac{S}{I}\longrightarrow \frac{S}{(I,x_3)}\longrightarrow 0
\end{align*}
It follows from depth lemma \cite[Proposition 1.2.9]{bh} that
\[
\begin{array}{rl}
{\rm depth}(S/I)\geq \min\big\{{\rm depth}(S/(I:x_3)), {\rm depth}(S/(I,x_3))\big\}.
\end{array} \tag{2} \label{2}
\]
Since $(I,x_3)=(u_{n-3},x_3)$, we have
\[
\begin{array}{rl}
{\rm depth}(S/(I,x_3))=n-2\geq 4.
\end{array} \tag{3} \label{3}
\]
On the other hand, notice that$$(I:x_3)=(x_2x_4, x_2x_6)+(x_1x_{i+4}\mid 1\leq i\leq n-4).$$In particular, there is a tree $T$ with vertex set $[n]\setminus\{3\}$ such that $(I:x_3)=I(T)$. It is easy to see that $\{1,2\}$ is a maximal independent set in $T$ of minimum size. Since $3$ is not a vertex of $T$, Lemma \ref{tree} implies that
\[
\begin{array}{rl}
{\rm depth}(S/(I:x_3))=2+1=3.
\end{array} \tag{4} \label{4}
\]
We conclude from inequalities (\ref{2}), (\ref{3}) and (\ref{4}) that ${\rm depth}(S/I)\geq 3$. This inequality together with inequality (\ref{1}) implies that ${\rm depth}(S/I)=3$. Equivalently, $g_I(1)=1$.

(ii) It is obvious that $I^{[2]}$ is the principal ideal generated by $u_{n-3}u_{n-1}$. Thus, ${\rm depth}(S/I^{[2]})=n-1$. In other words, $g_I(2)=n-6$.
\end{proof}

\begin{rem}
Note that for the ideal in Theorem \ref{main}, we have $\nu(I)$. Thus, Theorem \ref{main} shows that in general the function $g_I(k)$ can be an increasing function. However, we do not have any example of a graph $G$ for which the function $g_{I(G)}(k)$ is not nonincreasing. So, the conjecture posed in \cite{ehhs} might be true for edge ideals.
\end{rem}

%%%%%%%%%%%%%%%%%%%%%%%%%%%%%%%%%%%%%%%%%%%%%%%%%%%%%%%%%%%%%%%%%%%%%%%%%

\section*{Acknowledgment}

The author would like to thank the referee for a careful reading of the
paper and for valuable comments. This research is supported by the FAPA grant from the Universidad de los Andes.

%%%%%%%%%%%%%%%%%%%%%%%%%%%%%%%%%%%%%%%%%%%%%%%%%%%%%%%%%%%%%%%%%%%%%%%%%%

\section*{Declarations}

The author declares that there is no conflict of interest for this work.

%%%%%%%%%%%%%%%%%%%%%%%%%%%%%%%%%%%%%%%%%%%%%%%%%%%%%%%%%%%%%%%%%%%%%%%%%%

%%%%%%%%%%%%%%%%%%%%%%%%%%%%%%%%%%%%%%%%%%%%%%%%%%%%%%%%%%%%%%%%%%%%%%%%%%

\end{document}